\documentclass[12pt]{article}
\usepackage{amsmath}
\usepackage{amssymb}
\usepackage{amsthm}
\usepackage{extarrows}
\allowdisplaybreaks
\usepackage{tikz-cd}
\usepackage{graphicx}
\graphicspath{{/}}

\newtheorem{theorem}{Theorem}
\newtheorem*{theorem*}{Theorem}

\newtheorem{lemma}{Lemma}

\newtheorem{proposition}{Proposition}
\newtheorem*{proposition*}{Proposition}

\newcommand{\Mk}{\mathcal{M}_k}
\newcommand{\Mkk}{\mathcal{M}_{k-1}}
\newcommand{\Hk}{\mathcal{H}_k}
\newcommand{\Rm}{\mathbb{R}^m}
\newcommand{\C}{\mathbb{C}}
\newcommand{\Clm}{\mathcal{C}l_m}

\newcommand{\Sm}{\mathbb{S}^{m-1}}

\newcommand{\Eu}{\mathbb{E}_u}

\newcommand{\udx}{\langle u,D_x\rangle}
\newcommand{\dudx}{\langle D_u,D_x\rangle}

\newcommand{\Dtwo}{\mathcal{D}_2}

\newcommand{\be}{\begin{eqnarray*}}
\newcommand{\ee}{\end{eqnarray*}}

\begin{document}

\title{Some Properties of the Higher Spin Laplace Operator}

\author{Chao Ding$^1$\thanks{Electronic address:  {\tt dchao@uark.edu}.} and John Ryan$^1$\thanks{Electronic address: {\tt jryan@uark.edu.}} \\
\emph{\small $^1$Department of Mathematics, University of Arkansas, Fayetteville, AR 72701, USA} }
\date{}

\maketitle

\begin{abstract}
The higher spin Laplace operator has been constructed recently as the generalization of the Laplacian in higher spin theory. This acts on functions taking values in arbitrary irreducible representations of the Spin group. In this paper, we first provide a decomposition of the higher spin Laplace operator in terms of Rrita-Schwinger operators. With such a decomposition, a connection between the fundamental solutions for the higher spin Laplace operator and the fundamental solutions for the Rarita-Schwinger operators is provided. Further, we show that the two components in this decomposition are conformally invariant differential operators. An alternative proof for the conformally invariance property is also pointed out, which can be connected to Knapp-Stein intertwining operators. Last but not least, we establish a Borel-Pompeiu type formula for the higher spin Laplace operator. As an application, we give a Green type integral formula.
\end{abstract}
\section{Introduction}
Classical Clifford analysis (see \cite{Brackx,Belgians}) is the study of Dirac type operators and theory of monogenic functions (null solutions of the Dirac operator). It is also well known that (\cite{Brackx,Belgians}) solutions for the Dirac equation, $D_xf(x)=0$, in $m$-dimensional Euclidean space $\Rm$ can be described by a Cauchy integral formula. Here, $D_x$ stands for the Euclidean Dirac operator $\sum_{i=1}^me_i\frac{\partial}{\partial_{x_i}}$, and the $e_i$'s are the generators of a real Clifford algebra $\Clm$. As $D_x^2=-\Delta_x$, the negative Laplacian over $\Rm$, it has been shown that (\cite{JohnCauchy}) Green's formula for harmonic functions can be modified via Clifford algebras to more closely resemble a Cauchy integral formula. Further, many authors have also studied Cauchy integral formulas for modified Dirac equations. For instance, in \cite{Xu}, Xu studied solutions to the inhomogeneous Dirac equation $(D_x+\lambda)f(x)=0$, with $\lambda\in\mathbb{C}$, which also possess a Cauchy integral formula. In \cite{R1}, Ryan provided Cauchy kernels and Cauchy-Green type integral formulas for solutions to each polynomial equation $(D_x^n+\sum_{k=1}^{n-1}b_kD_x^k)f(x)=0,$ with $b_k\in\mathbb{C}$. \\
\par
In the past few decades, many authors (\cite{B1,Bures,Ding0,Ding3,D,Eelbode}) have been working on generalizations of classical Clifford analysis to the so called higher spin theory. This investigates higher spin differential operators acting on functions on $\Rm$, taking values in arbitrary irreducible representations of the Spin group. In Clifford analysis, these irreducible representations are traditionally constructed as homogeneous polynomial spaces satisfying certain differential equations, see \cite{Gilbert} for more details. The first order conformally invariant differential operator, named as a Rarita-Schwinger operator, was first studied systematically in \cite{Bures} and revisited in \cite{D} with a different approach. In both papers, fundamental solutions and integral formulas, such as Borel-Pompeiu formula and Cauchy integral formula, are established. Stein and Weiss (\cite{ES}) introduced Stein Weiss type gradients (also called Stein Weiss operators) as projections of the gradient operator with the language of representation theory. It turns out that the Dirac operator and the Rarita-Schwinger operator can both be reconstructed as Stein Weiss type gradients, see \cite{Ding0,SR,ES} for more details. In \cite{B1,Eelbode}, the second order conformally invariant differential operator, named as a generalized Maxwell operator or the higher spin Laplace operator, and its fundamental solution were discovered. In \cite{Ding3}, the authors completed the construction of arbitrary order  conformally invariant differential operators in higher spin spaces as well as their fundamental solutions.\\
\par
However, we have not found any results of integral formulas for these higher spin conformally invariant differential operators, except for the Rarita-Schwinger operator. In this paper, our purpose is to establish a Borel-Pompeiu type formula for the higher spin Laplace operator. The technique used here is motivated by the proof of the Green integral formula for harmonic functions via Clifford analysis. This suggests us to find the connection between the higher spin Laplace operator and the Rarita-Schwinger operators. Then we can investigate a Borel-Pompeiu type formula for the higher spin Laplace operator with the help of Stokes' theorem for the Rarita-Schwinger operators (\cite{Bures,Ding0}). A Green type integral formula is provided at the end of this paper.\\
\par
This paper is organized as follows. In Section 2, we introduce the notion of Clifford analysis and some of the standard facts in higher spin theory. For instance, we review Rairta-Schwinger operators and the higher spin Laplace operator with their fundamental solutions. Further, intertwining operators for the higher spin Laplace operator are also provided for later use. Section 3 is devoted to a brief summary of Stokes' theorem for the Rarita-Schwinger operators. This is one of the main tools for establishing integral formulas for the higher spin Laplace operator. Section 4 establishes the relation between the higher spin Laplace operator with the Rarita-Schwinger operators, this allows us to apply Stokes' theorem for the Rarita-Schwinger operators in the later argument. In Section 5, we also point out the connection between the fundamental solution of the higher spin Laplace operator and the fundamental solutions of two Rarita-Schwinger type operators, which turn out to be critical in the argument for the Borel-Pompeiu type formula in Section 7. Section 6 demonstrates that the decomposition for the higher spin Laplace operator, which is obtained in Section 4, provides us two second order conformally invariant differential operators. A straightforward proof is provided, and we also point out that there is an alternative proof. This needs the help of convolution type operators and the fundamental solutions for the higher spin Laplace operator. Similar argument can be found in Section $4.1$ in \cite{Ding3}. Such convolution type operators can also be realized as Knapp-Stein intertwining operators with the principal series representations induced by the polynomial representations of the Spin group. However, this is beyond the scope of this paper. See \cite{Clerc,KS} for more details.\\
\par
 With all the preparations that have been made, we establish a Borel-Pompeiu type formula for the higher spin Laplace operator in Section 7. As an application, a Green type integral formula is stated at the end.

\section{Preliminaries}
Let $\{e_1,e_2,\cdots,e_m\}$ be an orthonormal basis for the $m$-dimensional Euclidean space $\Rm$. The real Clifford algebra is generated by these basis elements with the defining relations $$e_i e_j + e_j e_i= -2\delta_{ij},\ 1\leq i,\ j\leq m,
$$ where $\delta_{ij}$ is the Kronecker delta function. An arbitrary element of the basis of the Clifford algebra can be written as $e_A=e_{j_1}\cdots e_{j_r},$ where $A=\{j_1, \cdots, j_r\}\subset \{1, 2, \cdots, m\}$ and $1\leq j_1< j_2 < \cdots < j_r \leq m.$
Hence for any element $a\in \mathcal{C}l_m$, we have $a=\sum_Aa_Ae_A,$ where $a_A\in \mathbb{R}$. The complex Clifford algebra $\Clm (\C)$ is defined as the complexification of the real Clifford algebra
$$\Clm (\C)=\Clm\otimes_{\mathbb{R}}\C.$$
We consider real Clifford algebra $\Clm$ throughout this subsection, but in the rest of the paper we consider the complex Clifford algebra $\Clm (\mathbb{C})$ unless otherwise specified. For $a=\sum_Aa_Ae_A\in\Clm$, we will need the following anti-involutions:
\begin{itemize}
\item \textbf{Reversion:}\\
\begin{eqnarray*}
\tilde{a}=\sum_{A}(-1)^{|A|(|A|-1)/2}a_Ae_A,
\end{eqnarray*}
where $|A|$ is the cardinality of $A$. In particular, $\widetilde{e_{j_1}\cdots e_{j_r}}=e_{j_r}\cdots e_{j_1}$. Also $\tilde{ab}=\tilde{b}\tilde{a}$ for $a, b\in\Clm.$
\item \textbf{Clifford conjugation:}\\
\begin{eqnarray*}
\bar{a}=\sum_{A}(-1)^{|A|(|A|+1)/2}a_Ae_A,
\end{eqnarray*}
satisfying $\overline{e_{j_1}\cdots e_{j_r}}=(-1)^re_{j_r}\cdots e_{j_1}$. Also $\overline{ab}=\bar{b}\bar{a}$ for $a, b\in\Clm$.
\end{itemize}
The Pin and Spin groups play an important role in Clifford analysis. The Pin group can be defined as $$Pin(m)=\{a\in \mathcal{C}l_m: a=y_1y_2\dots y_{p},\ y_1,\dots,y_{p}\in\mathbb{S}^{m-1},p\in\mathbb{N}\},$$ 
where $\mathbb{S} ^{m-1}$ is the unit sphere in $\Rm$. $Pin(m)$ is clearly a multiplicative group in $\mathcal{C}l_m$. \\
\par
Now suppose $a\in \mathbb{S}^{m-1}\subseteq \mathbb{R}^m$. If we consider $axa$, we may decompose
$$x=x_{a\parallel}+x_{a\perp},$$
where $x_{a\parallel}$ is the projection of $x$ onto $a$ and $x_{a\perp}$ is the remainder part perpendicular to $a$. Hence $x_{a\parallel}$ is a scalar multiple of $a$ and we have
$$axa=ax_{a\parallel}a+ax_{a\perp}a=-x_{a\parallel}+x_{a\perp}.$$
So the action $axa$ describes a reflection of $x$ in the direction of $a$. By the Cartan-Dieudonn$\acute{e}$ Theorem each $O\in O(m)$ is the composition of a finite number of reflections. If $a=y_1\cdots y_p\in Pin(m),$ we have $\tilde{a}:=y_p\cdots y_1$ and observe $ax\tilde{a}=O_a(x)$ for some $O_a\in O(m)$. Choosing $y_1,\ \dots,\ y_p$ arbitrarily in $\mathbb{S}^{m-1}$, we have the group homomorphism
\begin{eqnarray*}
\theta:\ Pin(m)\longrightarrow O(m)\ :\ a\mapsto O_a,
\end{eqnarray*}
with $a=y_1\cdots y_p$ and $O_ax=ax\tilde{a}$ is surjective. Further $-ax(-\tilde{a})=ax\tilde{a}$, so $1,\ -1\in Ker(\theta)$. In fact $Ker(\theta)=\{1,\ -1\}$. See \cite{P1}. The Spin group is defined as
$$Spin(m)=\{a\in \mathcal{C}l_m: a=y_1y_2\dots y_{2p},\ y_1,\dots,y_{2p}\in\mathbb{S}^{m-1},p\in\mathbb{N}\}$$
 and it is a subgroup of $Pin(m)$. There is a group homomorphism
\begin{eqnarray*}
\theta:\ Spin(m)\longrightarrow SO(m)
\end{eqnarray*}
that is surjective with kernel $\{1,\ -1\}$ and defined by the above group homomorphism for $Pin(m)$. Thus $Spin(m)$ is the double cover of $SO(m)$. See \cite{P1} for more details.\\
\par
The Dirac operator in $\mathbb{R}^m$ is defined to be $$D_x:=\sum_{i=1}^{m}e_i\partial_{x_i}.$$  Note $D_x^2=-\Delta_x$, where $\Delta_x$ is the Laplacian in $\mathbb{R}^m$.  A $\Clm$-valued function $f(x)$ defined on a domain $U$ in $\Rm$ is left monogenic if $D_xf(x)=0.$ Sometimes, we will consider the Dirac operator $D_u$ in a vector $u$ rather than $x$.\\ 
\par 
Let $\mathcal{M}_k$ denote the space of $\mathcal{C}l_m$-valued monogenic polynomials homogeneous of degree $k$. Note that if $h_k(u)\in\Hk$, the space of complex valued harmonic polynomials homogeneous of degree $k$, then $D_uh_k(u)\in\mathcal{M}_{k-1}$, but $D_uup_{k-1}(u)=(-m-2k+2)p_{k-1}u,$ where $p_{k-1}(u)\in \Mkk$. Hence,
\begin{eqnarray}\label{Almansi}
\mathcal{H}_k=\mathcal{M}_k\oplus u\mathcal{M}_{k-1},\ h_k=p_k+up_{k-1}.
\end{eqnarray}
This is an \emph{Almansi-Fischer decomposition} of $\Hk$ \cite{D}. In this Almansi-Fischer decomposition, we have $P_k^+$ and $P_k^-$ as the projection maps 
\begin{eqnarray}
&&P_k^+=1+\frac{uD_u}{m+2k-2}:\ \mathcal{H}_k\longrightarrow \mathcal{M}_k, \label{Pk+}\\
&&P_k^-=I-P_k^+=\frac{-uD_u}{m+2k-2}:\ \mathcal{H}_k\longrightarrow u\mathcal{M}_{k-1} \label{Pk-}.
\end{eqnarray}
Suppose $U$ is a domain in $\mathbb{R}^m$. Consider a differentiable function $f: U\times \mathbb{R}^m\longrightarrow \mathcal{C}l_m$
such that, for each $x\in U$, $f(x,u)$ is a left monogenic polynomial homogeneous of degree $k$ in $u$. Then \textbf{the Rarita-Schwinger operator} \cite{Bures,D} is defined by 
 $$R_k=P_k^+D_x:\ C^{\infty}(\Rm,\Mk)\longrightarrow C^{\infty}(\Rm,\Mk).$$
 We also need the following three more Rarita-Schwinger type operators.
 \begin{eqnarray*}
	&&\text{\textbf{The twistor operator:}}\ T_k=P_k^+D_x:\ C^{\infty}(\Rm,u\Mkk)\longrightarrow C^{\infty}(\Rm,\Mk),\\
	&&\text{\textbf{The dual twistor operator:}}\ T_k^*=P_k^-D_x:\ C^{\infty}(\Rm,\Mk)\longrightarrow C^{\infty}(\Rm,u\Mkk),\\
	&&\text{\textbf{The remaining operator:}}\ Q_k=P_k^-D_x:\ C^{\infty}(\Rm,u\Mkk)\longrightarrow C^{\infty}(\Rm,u\Mkk).
	\end{eqnarray*}
	More details can be found in \cite{Bures,D}. Let $Z_k^1(u,v)$ be the reproducing kernel for $\Mk$, which satisfies
\begin{eqnarray*}
f(v)=\int\displaylimits_{\Sm}\overline{Z_k^1(u,v)}f(u)dS(u),\ for\ all\ f(v)\in\Mk.
\end{eqnarray*}
Then the fundamental solution for $R_k$ (\cite{D}) is 
\begin{eqnarray*}
E_k(x,u,v)=\frac{1}{\omega_{m}a_k}\frac{x}{||x||^m}Z_k^1(\frac{xux}{||x||^2},v),
\end{eqnarray*}
where the constant $a_k$ is $\displaystyle\frac{m-2}{m+2k-2}$ and $\omega_{m}$ is the area of the $m$-dimensional unit sphere. Similarly, we have the fundamental solution for $Q_k$ (\cite{Li}) as follows.
\begin{eqnarray*}
F_k(x,u,v)=\frac{-1}{\omega_{m}a_k}u\frac{x}{||x||^m}Z_{k-1}^1(\frac{xux}{||x||^2},v)v.
\end{eqnarray*}
The higher spin Laplace operator is constructed in \cite{B1}, and is defined as follows.
\begin{eqnarray*}
\Dtwo=\Delta_x-\displaystyle\frac{4\udx\dudx}{m+2k-2}+\displaystyle\frac{4||u||^2\dudx^2}{(m+2k-2)(m+2k-4)}.
\end{eqnarray*}
where $\langle\ ,\ \rangle$ is the standard inner product in Euclidean space. The fundamental solution for $\Dtwo$ is also provided in the same reference. We denote it by
\begin{eqnarray*}
H_k(x,u,v)=\displaystyle\frac{(m+2k-4)\Gamma(\displaystyle\frac{m}{2}-1)}{4(4-m)\pi^{\frac{m}{2}}}||x||^{2-m}Z_k^2(\displaystyle\frac{xux}{||x||^2},v),
\end{eqnarray*}
where $Z_k^2(u,v)$ is the reproducing kernel for $\Hk$ and satisfies
\begin{eqnarray*}
f(v)=\int\displaylimits_{\Sm}\overline{Z_k^2(u,v)}f(u)dS(u),\ for\ all\ f(v)\in\Hk.
\end{eqnarray*} 
The result of intertwining operators for $\Dtwo$ is stated as follows for later use.
\begin{theorem}\cite{B1,Ding3}\label{interDtwo}
Let $f(x,u)\in C^{\infty}(\Rm,\Hk)$, $x'=\varphi(x)=(ax+b)(cx+d)^{-1}$ is a M\"{o}bius transformation and $u'=\displaystyle\frac{(cx+d)u(\widetilde{cx+d})}{||cx+d||^2}$. Then we have
\begin{eqnarray*}
&&J_{-2}^{-1}\mathcal{D}_{2,x,u}J_2(\varphi,x)f(\varphi(x),\frac{(cx+d)u(\widetilde{cx+d})}{||cx+d||^2})=\mathcal{D}_{2,x',u'}f(x',u').
\end{eqnarray*}
where $J_{-2}(\varphi,x)=||cx+d||^{-m-2}$ and $J_{2}(\varphi,x)=||cx+d||^{2-m}$ are called intertwining operators for $\Dtwo$. The lower index in $\mathcal{D}_{2,x,u}$ stands for the variables of the operator $\Dtwo$.
\end{theorem}
To conclude this section, we point out that, if $D_uf(u)=0$ then $\bar{f}(u)\bar{D}_u=-\bar{f}(u)D_u=0$. So we can talk of right monogenic functions, right monogenic polynomials with homogeneity of degree $k$, right Almansi-Fischer decomposition for $\Hk$, etc. In other words, for all the results we introduced above, we have their analogues for right monogenic functions. 
\section{Stokes' theorem for the Rarita-Schwinger type operators}
For the convenience of the reader, we review Stokes' theorem for the Rarita-Schwinger type toperators as follows. More details can be found in \cite{Bures,Ding0}.

\begin{theorem}[\cite{D}]\textbf{(Stokes' theorem for $R_k$)}\label{StokesRk}\\
Let $\Omega '$ and $\Omega$ be domains in $\mathbb{R}^m$ and suppose the closure of $\Omega$ lies in $\Omega '$. Further suppose the closure of $\Omega$ is compact and $\partial \Omega$ is piecewise smooth. Let $f,g\in C^1(\Omega ',\Mk)$. Then
\newpage
\begin{eqnarray*}
&&\int\displaylimits_{\Omega}\big[(g(x,u)R_k,f(x,u))_u+(g(x,u),R_kf(x,u))_u\big]dx^m\\
=&&\int\displaylimits_{\partial\Omega}(g(x,u),d\sigma_xf(x,u))_u.
\end{eqnarray*}
where $d\sigma_x=n(x)d\sigma(x)$, $d\sigma(x)$ is the area element. $(P(u),Q(u))_u=\int\displaylimits_{\mathbb{S}^{m-1}}P(u)Q(u)dS(u)$ is the inner product for any pair of $\Clm$-valued polynomials.
\end{theorem}
Comparing with the version in \cite{D}, we shortened the identity by missing $P_k^+$ in the last two equations on purpose. This can also be found in the proof of this theorem in \cite{D}. We also want to clarify that, in the Stokes' theorem above, for $g(x,u)R_k$, the $R_k$ is the right Rarita-Schwinger operator. Since it is positioned on the right hand side, there should be no confusion for this. The reader will see more such terms in the rest of this paper.
\begin{theorem}[\cite{Ding0}]\textbf{(Stokes' Theorem for $T_k$ and $T_k^*$)}\label{StokesTk}\\
Let $\Omega$ and $\Omega '$ be defined as above. Then for $f\in C^1(\Rm,\Mk)$ and $g\in C^1(\Rm,u\Mkk)$, we have
\begin{eqnarray*}
&&\int\displaylimits_{\partial\Omega}\big(g(x,u),d\sigma_xf(x,u)\big)_u\\
&=&\int\displaylimits_{\Omega}\big(g(x,u)T_k,f(x,u)\big)_udx^m+\int\displaylimits_{\Omega}\big(g(x,u),T_k^*f(x,u)\big)_udx^m.
\end{eqnarray*}
\end{theorem}
\begin{theorem}(\cite{Li})\textbf{(Stokes' theorem for $Q_k$)}\\
Let $\Omega$ and $\Omega '$ be defined as above. Then for $f,g\in C^1(\Rm,u\Mkk)$, we have
\begin{eqnarray*}
&&\int\displaylimits_{\Omega}\big[(g(x,u)Q_k,f(x,u))_u+(g(x,u),Q_kf(x,u))_u\big]dx^m\\
=&&\int\displaylimits_{\partial\Omega}(g(x,u),d\sigma_xf(x,u))_u.
\end{eqnarray*}
\end{theorem}
We also missed $P_k^-$ in the last equation for convenience. See \cite{Li}.
 
\section{Connection between $\Dtwo$ and the Rarita-Schwinger operators}
In this section, we will rewrite the higher spin Laplace operator $\Dtwo$ in terms of the Rarita-Schwinger operators. This allows us to apply the Stokes' theorem for the Rarita-Schwinger operators in the last section.
\begin{proposition}\label{D2exp}
Let $P_k^+$ and $P_k^-$ be the projection maps defined in (\ref{Pk+}) and (\ref{Pk-}). Then the higher spin Laplace operator $\Dtwo$ can be written as follows.
\begin{eqnarray}
\Dtwo&=&-R_k^2P_k^++\frac{2T_k^*R_kP_k^+}{m+2k-4}-\frac{2T_kQ_kP_k^-}{m+2k-4}-\frac{(m+2k)Q_k^2P_k^-}{m+2k-4},\label{D21}\\
         &=&-R_k^2P_k^++\frac{2R_kT_kP_k^-}{m+2k-4}-\frac{2Q_kT_k^*P_k^+}{m+2k-4}-\frac{(m+2k)Q_k^2P_k^-}{m+2k-4},\label{D22}
\end{eqnarray}
when it acts on a function $f(x,u)\in C^{\infty}(\Rm,\Hk).$
\end{proposition}
Here, we first prove the equation (\ref{D21}). To accomplish this, we need the following lemma.
\begin{lemma} (\cite{B1})\label{lemmaD2}
Let $p_1$ and $p_0$ be the projection maps in the Almansi-Fischer decomposition of $\Hk$ defined as
\begin{eqnarray*}
p_1=1+\frac{uD_u}{m+2k-2}:\ \Hk\longrightarrow \Mk,\\
p_0=\frac{-D_u}{m+2k-2}:\ \Hk\longrightarrow \Mkk.
\end{eqnarray*}
Then for any $f(x,u)\in C^{\infty}(\Rm,\Hk)$, we have
\begin{eqnarray}\label{D23}
\Dtwo&=&-R_k^2p_1+\frac{4u\langle D_u,D_x\rangle R_kp_1}{(m+2k-2)(m+2k-4)}p_1\nonumber\\
&&-uR_{k-1}^2p_0-\frac{4}{m+2k-2}\big(\langle u,D_x\rangle-\frac{||u||^2\langle D_u,D_x\rangle}{m+2k-4}\big)R_{k-1}p_0.
\end{eqnarray}
\end{lemma}
Here, we remind the reader that the projection map $p_1$ is the same as $P_k^+$, but $p_0$ is different from $P_k^-$. Indeed, $P_k^-=up_0$. Now, we can start verifying equation (\ref{D21}).
\begin{proof}
From the expression of $\Dtwo$ in Lemma \ref{lemmaD2}, we notice that, for any $f(x,u)\in C^{\infty}(\Rm,\Hk)$, we have $R_kp_1f(x,u)\in C^{\infty}(\Rm,\Mk)$. It is easy to see that
\begin{eqnarray*}
u\dudx R_kp_1f(x,u)=-\frac{uD_uD_x+uD_xD_u}{2}R_kp_1f(x,u)=-\frac{uD_uD_x}{2}R_kp_1f(x,u).
\end{eqnarray*}
The last equation comes from $D_u R_kp_1f(x,u)=0$. Hence, the first two terms of (\ref{D23}) become
\begin{eqnarray*}
-R_k^2p_1+\frac{4u\langle D_u,D_x\rangle R_k}{(m+2k-2)(m+2k-4)}p_1
=-R_k^2P_k^++\frac{2T_k^*R_kP_k^+}{m+2k-4},
\end{eqnarray*}
which are exactly the first two terms in (\ref{D21}). \\
\par
On the other hand, Lemma $\ref{lemmaD2}$ tells that the last two terms in (\ref{D21}) come from $\Dtwo$ acting on $P_k^-f(x,u)\in C^{\infty}(\Rm, u\Mkk)$, for $f(x,u)\in C^{\infty}(\Rm,\Hk)$. Next, we show that $\Dtwo$ is a linear combination of $Q_k^2$ and $T_kQ_k$ when acting on functions in $C^{\infty}(\Rm,u\Mkk)$. This makes sense, since $Q_k^2$ and $T_kQ_k$ are the two ``paths", which start from $C^{\infty}(\Rm,u\Mkk)$ and end in $C^{\infty}(\Rm,\Hk)$. See the following diagram.
\begin{center}
\includegraphics{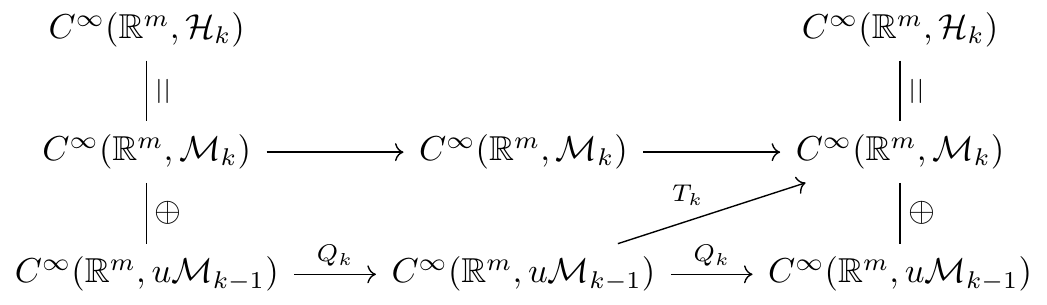}
\end{center}
Hence, we calculate $\Dtwo$, $Q_k^2$ and $T_kQ_k$ acting on $C^{\infty}(\Rm,u\Mkk)$, respectively. For our convenience, we assume $P_k^-f(x,u)=ug(x,u)$ for some $g(x,u)\in C^{\infty}(\Rm,\Mkk)$. Now, we calculate $\Dtwo u$ and keep in mind that it acts on a function $g(x,u)\in\Mkk$ with respect to the variable $u$. In other words, in the calculation below, any operator with $D_u$ on the right hand side vanishes.
\begin{eqnarray}
\Dtwo u&=&\bigg(\Delta_x-\displaystyle\frac{4}{m+2k-2}\big(\udx-\displaystyle\frac{||u||^2\dudx^2}{m+2k-4}\big)\dudx\bigg)u\nonumber\\
            &=&u\Delta_x-\frac{4}{m+2k-2}\big(\udx-\frac{||u||^2\dudx}{m+2k-4}\big)\big(u\dudx+D_x\big)\nonumber\\
            &=&u\Delta_x-\frac{4u\udx\dudx}{m+2k-2}-\frac{4\udx D_x}{m+2k-2}+\frac{4u||u||^2\dudx^2}{(m+2k-2)(m+2k-4)}\nonumber\\
            &&+\frac{8||u||^2\dudx D_x}{(m+2k-2)(m+2k-4)}.\label{Dtwou}
\end{eqnarray}
On the other hand, since
$$D_xu=-uD_x-2\langle u,D_x\rangle,\ D_uu=-m-2\Eu-uD_u,$$
we have
\begin{eqnarray*}
&&Q_k^2u=\frac{uD_uD_xuD_uD_x}{(m+2k-2)^2}u=\frac{uD_uD_xuD_u(-uD_x-2\udx)}{(m+2k-2)^2}\\
&=&\frac{1}{(m+2k-2)^2}\big(-uD_uD_xu(-m-uD_u-2\mathbb{E}_u)D_x-2uD_uD_xu(\udx D_u+D_x)\big),
\end{eqnarray*}
where $\mathbb{E}_u=\displaystyle\sum_{i=1}^mu_i\displaystyle\frac{\partial}{\partial u_i}$ is the Euler operator. Since $\Eu D_xg(x,u)=(k-1)D_xg(x,u)$ and $D_ug(x,u)=0$ for $g(x,u)\in C^{\infty}(\Rm,\Mkk)$. The equation above becomes
\begin{eqnarray*}
&&\frac{m}{(m+2k-2)^2}uD_uD_xuD_x-\frac{uD_uD_x||u||^2D_uD_x}{(m+2k-2)^2}+\frac{2uD_uD_xu\mathbb{E}_uD_x}{(m+2k-2)^2}-\frac{2uD_uD_xuD_x}{(m+2k-2)^2}\\
&=&\frac{m+2k-4}{m+2k-2}uD_u(-uD_x-2\udx)D_x-\frac{u(2u+||u||^2D_u)D_xD_uD_x}{(m+2k-2)^2}.
\end{eqnarray*}
The last equation comes from combining all $uD_uD_xuD_x$ terms with the fact that $\Eu D_x=(k-1)D_x$. Applying similar identities that we just used again, we get
\newpage
\begin{eqnarray}
&&-\frac{m+2k-4}{(m+2k-2)^2}u(-m-uD_u-2\mathbb{E}_u)D_x^2-\frac{2(m+2k-4)}{(m+2k-2)^2}u(D_x+\udx D_u)D_x\nonumber\\
&&-\frac{2u^2D_xD_uD_x}{(m+2k-2)^2}-\frac{u||u||^2}{(m+2k-2)^2}(-D_xD_u-2\dudx)(-2\dudx).\nonumber
\end{eqnarray}
Simplifying the previous equation, we have
\begin{eqnarray}
Q_k^2u&=&\frac{(m+2k-4)^2}{(m+2k-2)^2}uD_x^2+\frac{4(m+2k-4)}{(m+2k-2)^2}u\udx\dudx+\frac{4u^2D_x\dudx}{(m+2k-2)^2}\nonumber\\
&&-\frac{4u||u||^2\dudx^2}{(m+2k-2)^2}.\label{Qk2u}
\end{eqnarray}
Similarly, we have the details for rewriting $T_kQ_ku$ as follows. Since $D_xu=-uD_x-2\langle u,D_x\rangle,$ we have
\begin{eqnarray*}
&&T_kQ_ku=(D_x+\frac{uD_uD_x}{m+2k-2})(-\frac{uD_uD_x}{m+2k-2})u=-D_x\frac{uD_uD_x}{m+2k-2}u-Q_k^2u\\
&=&\frac{-1}{m+2k-2}(-uD_x-2\udx)D_u(-uD_x-2\udx)-Q_k^2u\\
&=&-\frac{uD_xD_uuD_x}{m+2k-2}-\frac{2\udx D_uuD_x}{m+2k-2}-\frac{2uD_xD_u\udx}{m+2k-2}-\frac{4\udx D_u\udx}{m+2k-2}-Q_k^2u.
\end{eqnarray*}
As $D_uu=-m-2\Eu-uD_u$ and $D_u\langle u,D_x\rangle=D_x+\langle u,D_x\rangle$, the equation above becomes
\begin{eqnarray*}
&&-\frac{uD_x(-m-2\Eu-uD_u)D_x}{m+2k-2}-\frac{2\udx(-m-2\Eu-uD_u)D_x}{m+2k-2}\\
&&-\frac{2uD_x(\udx D_u+D_x)}{m+2k-2}-\frac{4\udx(\udx D_u+D_x)}{m+2k-2}-Q_k^2u.
\end{eqnarray*}
Recall that this operator acts on $g(x,u)\in C^{\infty}(\Rm,\Mkk)$, which means $$\Eu D_xg(x,u)=(k-1)D_xg(x,u),\ uD_x\udx D_ug(x,u)=0\ \text{and}\ \udx^2 D_ug(x,u)=0.$$ 
Hence, with $D_uD_x=-D_xD_u-2\langle D_u,D_x\rangle$, we get
\begin{eqnarray*}
T_kQ_ku&=&\frac{uD_xuD_uD_x}{m+2k-2}+uD_x^2+2\udx D_x-\frac{4u\udx\dudx}{m+2k-2}-\frac{2uD_x^2}{m+2k-2}\\
&&-\frac{4\udx D_x}{m+2k-2}-Q_k^2u\\
&=&\frac{u(-uD_x-2\udx)(-D_xD_u-2\dudx)}{m+2k-2}+\frac{m+2k-4}{m+2k-2}uD_x^2\\
&&+\frac{2(m+2k-4)}{m+2k-2}\udx D_x-\frac{4u\udx\dudx}{m+2k-2}-Q_k^2u
\end{eqnarray*}
Since $u(-uD_x-2\udx)D_xD_ug(x,u)=0$ for $g(x,u)\in C^{\infty}(\Rm,\Mkk)$, the previous equation becomes
\begin{eqnarray}
&&\frac{-2||u||^2\dudx D_x+4u\udx\dudx}{m+2k-2}+\frac{m+2k-4}{m+2k-2}uD_x^2\nonumber\\
&&+\frac{2(m+2k-4)}{m+2k-2}\udx D_x-\frac{4u\udx\dudx}{m+2k-2}-Q_k^2u.\nonumber
\end{eqnarray}
Simplifying the equation above, we have
\begin{eqnarray}
T_kQ_ku=-\frac{2||u||^2\dudx D_x}{m+2k-2}+\frac{m+2k-4}{m+2k-2}uD_x^2+\frac{2(m+2k-4)}{m+2k-2}\udx D_x-Q_k^2u.\label{TkQku}
\end{eqnarray}
Combining the equations (\ref{Dtwou}), (\ref{Qk2u}) and (\ref{TkQku}), a straightforward verification shows that
\begin{eqnarray*}
\Dtwo u=-\frac{2T_kQ_ku}{m+2k-4}-\frac{m+2k}{m+2k-4}Q_k^2u.
\end{eqnarray*}
This completes the proof of equation (\ref{D21}). In \cite{Bures}, the authors point out that
\begin{eqnarray*}
&&T_kQ_kf(x,u)=-R_kT_kf(x,u),\ \text{if}\ f(x,u)\in C^{\infty}(\Rm,u\Mkk);\\
&&T_k^*R_kf(x,u)=-Q_kT_k^*f(x,u),\ \text{if}\ f(x,u)\in C^{\infty}(\Rm,\Mk).
\end{eqnarray*} 
Then we can obtain (\ref{D22}) from (\ref{D21}) immediately.
\end{proof}

\section{Connection between fundamental solution of $\Dtwo$ and fundamental solutions of $R_k$ and $Q_k$}
In equation (\ref{D22}), we have
\be
\Dtwo&=&-R_k^2P_k^++\frac{2R_kT_kP_k^-}{m+2k-4}-\frac{2Q_kT_k^*P_k^+}{m+2k-4}-\frac{(m+2k)Q_k^2P_k^-}{m+2k-4}\\
&=&R_k(-R_kP_k^++\frac{2T_kP_k^-}{m+2k-4})+Q_k(-\frac{2T_k^*P_k^+}{m+2k-4}-\frac{(m+2k)Q_kP_k^-}{m+2k-4}).
\ee
We let $A_k=-R_kP_k^++\displaystyle\frac{2T_kP_k^-}{m+2k-4}$ and $B_k=-\displaystyle\frac{2T_k^*P_k^+}{m+2k-4}-\displaystyle\frac{(m+2k)Q_kP_k^-}{m+2k-4}$ for convenience, then we have
$$\Dtwo=R_kA_k+Q_kB_k.$$
The main result of this section is the following.
\begin{proposition}\label{fund}
Let $E_k(x,u,v),\ F_k(x,u,v)$ and $H_k(x,u,v)$ be the fundamental solutions for $R_k$, $Q_k$ and $\Dtwo$ respectively (see Section $2$). We have
\be
&&H_k(x,u,v)A_{k,r}=E_k(x,u,v),\\
&&H_k(x,u,v)B_{k,r}=F_k(x,u,v),
\ee
where $A_{k,r}$ and $B_{k,r}$ act from the right hand side of $H_k(x,u,v).$
\end{proposition}
To prove the proposition above, we need the following technical lemma.
\begin{lemma}\cite{D}\label{Ortho}
Suppose $p_k$ is a left monogenic polynomial homogeneous of degree $k$ and $p_{k-1}$ is a right monogenic polynomial homogeneous of degree $k-1$ then
\be
(p_{k-1}(u)u,p_k(u))_u=\int\displaylimits_{\Sm}p_{k-1}(u)up_k(u)dS(u)=0.
\ee
\end{lemma}
Now, we start proving the previous proposition.
\begin{proof}
Since $H_k(x,u,v)$ is the fundamental solution for $\Dtwo$, then we have
\be
 H_k(x-y,u,v)\Dtwo=\delta(y)Z_k^2(u,v),
\ee
where $Z_k^2(u,v)$ is the reproducing kernel for $\Hk$. In other words, for any function $f(x,u)\in C^{\infty}(\Rm,\Hk)$ with compact support with respect to the variable $x$, denoted by $ C_c^{\infty}(\Rm,\Hk)$, we have
\begin{eqnarray}
f(y,v)&=&\int\displaylimits_{R^m}(H_k(x-y,u,v)\Dtwo,f(x,u))_u\nonumber\\
&=&\int\displaylimits_{R^m}(H_k(x-y,u,v)(A_{k,r}R_k+B_{k,r}Q_k),P_k^+f(x,u)+P_k^-f(x,u))_u.\label{equation1}
\end{eqnarray}
From the expression of $A_{k,r}$, we observe that 
$$H_k(x-y,u,v)A_{k,r}\in C^{\infty}(\Rm,\mathcal{M}_{k,r}),\ H_k(x-y,u,v)A_{k,r}R_k\in C^{\infty}(\Rm,\mathcal{M}_{k,r}),$$
 where $\mathcal{M}_{k,r}$ stands for right monogenic polynomial space homogeneous of degree $k$. Similarly, we know that $H_k(x-y,u,v)B_{k,r}Q_k\in C^{\infty}(\Rm,\mathcal{M}_{k-1,r}u)$. In the meantime, we know that $P_k^+f\in C^{\infty}(\Rm,\Mk)$ and $P_k^-f\in C^{\infty}(\Rm,u\Mkk)$. Hence, with the help of Lemma \ref{Ortho}, we have
\begin{eqnarray}\label{RkAkreconstruct}
(H_k(x-y,u,v)A_{k,r}R_k,P_k^-f(x,u))_u=(H_k(x-y,u,v)B_{k,r}Q_k,P_k^+f(x,u))_u=0.
\end{eqnarray}
Therefore, equation (\ref{equation1}) tells us that
\be
&&\int\displaylimits_{R^m}(H_k(x-y,u,v)A_{k,r}R_k,P_k^+f(x,u))_u+\int\displaylimits_{R^m}(H_k(x-y,u,v)B_{k,r}Q_k,P_k^-f(x,u))_u\\
&&=f(y,v)=P_k^+f(y,v)+P_k^-f(y,v),
\ee
for any $f(y,v)\in C_c^{\infty}(\Rm,\Hk)$. In particular, for any $f(y,v)\in C_c^{\infty}(\Rm,\Mk)$, which means $P_k^-f=0$, we have
\be
\int\displaylimits_{R^m}(H_k(x-y,u,v)A_{k,r}R_k,P_k^+f(x,u))_u=P_k^+f(y,v).
\ee
This implies that 
$$H_k(x-y,u,v)A_{k,r}R_k=\delta(y)Z_k^1(u,v),$$
where $Z_k^1(u,v)$ is the reproducing kernel for $\Mk$. However, from the definition of fundamental solutions for a differential operator, this also implies that
$H_k(x-y,u,v)A_{k,r}$ must be the fundamental solution for $R_k$, i.e., 
$$H_k(x-y,u,v)A_k=E_k(x-y,u,v).$$
Similarly, if $f(y,v)\in C_c^{\infty}(\Rm,u\Mkk)$, which means $P_k^+f=0$, then we have
\begin{eqnarray}\label{QkBkreconstruct}
\int\displaylimits_{R^m}(H_k(x-y,u,v)B_{k,r}Q_k,P_k^-f(x,u))_u=P_k^-f(y,v),
\end{eqnarray}
which implies $H_k(x-y,u,v)B_{k,r}Q_k=\delta(x-y)Z_{k-1}^1(u,v)u$. We remind the reader that the $``u"$ on the right hand side is cancelled by the $``u"$ in $P_k^-f(x,u)\in C^{\infty}(\Rm,u\Mkk)$. Therefore,
$$H_k(x-y,u,v)B_{k,r}=F_k(x-y,u,v),$$
which completes the proof.
\end{proof}
\section{Conformally invariant property for $R_kA_k$ and $Q_kB_k$}
We notice that in the decomposition of $\Dtwo=R_kA_k+Q_kB_k$,
\begin{eqnarray}
&&R_kA_k:\ C^{\infty}(\Rm,\Hk)\longrightarrow C^{\infty}(\Rm,\Mk),\label{RkAk}\\
&&Q_kB_k:\ C^{\infty}(\Rm,\Hk)\longrightarrow C^{\infty}(\Rm,u\Mkk).\label{QkBk}
\end{eqnarray}
In this section, we will show that $R_kA_k$ and $Q_kB_k$ are both conformally invariant. More specifically,
\begin{theorem}\label{Intertwining}
Let $f(x,u)\in C^{\infty}(\Rm,\Hk)$, $x'=\varphi(x)=(ax+b)(cx+d)^{-1}$ is a M\"{o}bius transformation and $u'=\displaystyle\frac{(cx+d)u(\widetilde{cx+d})}{||cx+d||^2}$. Then we have
\begin{eqnarray*}
&&J_{-2}^{-1}(\varphi,x)(R_kA_k)_{x,u}J_2(\varphi,x)f(\varphi(x),\frac{(cx+d)u(\widetilde{cx+d})}{||cx+d||^2})=(R_kA_k)_{x',u'}f(x',u'),\\
&&J_{-2}^{-1}(\varphi,x)(Q_kB_k)_{x,u}J_2(\varphi,x)f(\varphi(x),\frac{(cx+d)u(\widetilde{cx+d})}{||cx+d||^2})=(Q_kB_k)_{x',u'}f(x',u'),
\end{eqnarray*}
where $J_{-2}(\varphi,x)$ and $J_{2}(\varphi,x)$ are defined in Theorem \ref{interDtwo}.
\end{theorem}
\begin{proof}
From Theorem \ref{interDtwo}, for $f(x,u)\in C^{\infty}(\Rm,\Hk)$, we have
\be
J_{-2}^{-1}(\varphi,x)\mathcal{D}_{2,x,u}J_2(\varphi,x)f(\varphi(x),\frac{(cx+d)u(\widetilde{cx+d})}{||cx+d||^2})=\mathcal{D}_{2,x',u'}f(x',u').
\ee
Since $\Dtwo=R_kA_k+Q_kB_k$, the previous equation becomes
\begin{eqnarray}
&&J_{-2}^{-1}(\varphi,x)(R_kA_k)_{x,u}J_2(\varphi,x)f(\varphi(x),\frac{(cx+d)u(\widetilde{cx+d})}{||cx+d||^2})\nonumber\\
&&+J_{-2}^{-1}(\varphi,x)(Q_kB_k)_{x,u}J_2(\varphi,x)f(\varphi(x),\frac{(cx+d)u(\widetilde{cx+d})}{||cx+d||^2})\nonumber\\
&=&(R_kA_k)_{x',u'}f(x',u')+(Q_kB_k)_{x',u'}f(x',u').\label{decomp}
\end{eqnarray}
From Theorem \ref{interDtwo}, we know that $J_{2}(\varphi,x)f(\varphi(x),\displaystyle\frac{(cx+d)u(\widetilde{cx+d})}{||cx+d||^2})\in C^{\infty}(\Rm,\Hk)$. Then with the help of (\ref{RkAk}) and (\ref{QkBk}), we obtain
\be
&&J_{-2}^{-1}(\varphi,x)(R_kA_k)_{x,u}J_2(\varphi,x)f(\varphi(x),\frac{(cx+d)u(\widetilde{cx+d})}{||cx+d||^2})\in C^{\infty}(\Rm,\Mk),\\
&&J_{-2}^{-1}(\varphi,x)(Q_kB_k)_{x,u}J_2(\varphi,x)f(\varphi(x),\frac{(cx+d)u(\widetilde{cx+d})}{||cx+d||^2})\in C^{\infty}(\Rm,u\Mkk),\\
&&(R_kA_k)_{x',u'}f(x',u')\in C^{\infty}(\Rm,\Mk),\ (Q_kB_k)_{x',u'}f(x',u')\in C^{\infty}(\Rm,u\Mkk).
\ee
Meanwhile, from the Almansi-Fischer decomposition of $\Hk$ (see (\ref{Almansi})), we know that 
$$C^{\infty}(\Rm,\Mk)\cap C^{\infty}(\Rm,u\Mkk)=\emptyset.$$ Therefore, in equation (\ref{decomp}), the functions on both sides, which belong to the same function space, are equal to each other. In other words,
\be
&&J_{-2}^{-1}(\varphi,x)(R_kA_k)_{x,u}J_2(\varphi,x)f(\varphi(x),\frac{(cx+d)u(\widetilde{cx+d})}{||cx+d||^2})=(R_kA_k)_{x',u'}f(x',u'),\\
&&J_{-2}^{-1}(\varphi,x)(Q_kB_k)_{x,u}J_2(\varphi,x)f(\varphi(x),\frac{(cx+d)u(\widetilde{cx+d})}{||cx+d||^2})=(Q_kB_k)_{x',u'}f(x',u'),
\ee
which completes the proof.
\end{proof}
We point out that there is an alternative approach for the conformally invariant property for $R_kA_k$ and $Q_kB_k$. We notice that equation (\ref{RkAkreconstruct}) and (\ref{QkBkreconstruct}) provide us integrals to reproduce $P_k^+f(y,v)$ and $P_k^-f(y,v)$. The integrals of the form
\be
\int\displaylimits_{\Rm}\big(E_k(x-y,u,v),f(x,u)\big)_u
\ee
are called convolution type operators generated by the fundamental solution $E_k(x,u,v)$ in \cite{Ding3}. These can be considered as Knapp-Stein intertwining operators with respect to a principal series representation of the Spin group. See \cite{Clerc,KS} for more details. Further, we have already seen the intertwining operators for $\Dtwo$ in Theorem {\ref{interDtwo}}. With similar argument as in Section $4.1$ in \cite{Ding3}, we can also have Theorem \ref{Intertwining}.
\section{Integral formulas for the higher spin Laplace operator}
With the decomposition of $\Dtwo$ obtained in Section $4$ and the Stokes' theorem for the Rarita-Schwinger operators, we will establish a Borel-Pompeiu type formula for the higher spin Laplace operator in this section.
\begin{theorem}\textbf{(Borel-Pompeiu type formula)}\label{BorelPompeiu}\\
Let $\Omega$ and $\Omega'$ be domains in $\Rm$ and suppose the closure of $\Omega$ lies in $\Omega'$. Further suppose the closure of $\Omega$ is compact and $\partial\Omega$ is piecewise smooth. Let $f(x,u)\in C^{\infty}(\Rm,\Hk)$ and $y\in\Omega$, then we have
\be
&&\int\displaylimits_{\partial\Omega}\big(H_k(x-y,u,v)(P_k^+-\frac{2P_k^-}{m+2k-4}),d\sigma_xR_kP_k^+f(x,u)\big)_u\\
&&+\int\displaylimits_{\partial\Omega}\big(E_k(x-y,u,v),d\sigma_xP_k^+f(x,u)\big)_u\\
&&+\int\displaylimits_{\partial\Omega}\big(H_k(x-y,u,v)(\frac{2P_k^+}{m+2k-4}+\frac{m+2k}{m+2k-4}P_k^-),d\sigma_xQ_kP_k^-f(x,u)\big)_u\\
&&+\int\displaylimits_{\partial\Omega}\big(F_k(x-y,u,v),d\sigma_xP_k^-f(x,u)\big)_u\\
&=&-\int\displaylimits_{\Omega}\big(H_k(x-y,u,v),\Dtwo f(x,u)\big)_udx^m+f(y,v),
\ee
where $dx^m=dx_1\wedge\cdots\wedge dx_m$, $d\sigma_x=n(x)d\sigma(x)$, $\sigma$ is scalar Lebesgue measure on $\partial\Omega$ and $n(x)$ is unit outer normal vector to $\partial\Omega$. 
\end{theorem}
Before we prove the theorem above, we remind the reader that, in the Stokes' Theorem for the Rarita-Schwinger type operators in Section $3$, the function spaces in these theorems are different. More specifically, for instance, the Stokes' theorem for $R_k$ requires that $f(x,u),g(x,u)\in C^1(\Rm,\Mk)$ but the Stokes' theorem for $Q_k$ requires that $f(x,u),g(x,u)\in C^1(\Rm,u\Mkk)$. However, in the theorem above, we notice
$$H_k(x-y,u,v),\ f(x,u)\in C^{\infty}(\Rm,\Hk).$$
Hence, in order to apply the Stokes' theorem for the Rarita-Schwinger type operators, we have to break them into $C^{\infty}(\Rm,\Mk)$ and $C^{\infty}(\Rm,u\Mkk)$ via the Almansi-Fischer decomposition for $\Hk$. That is why $P_k^+$ and $P_k^-$ are involved in the theorem.
\begin{proof}
We first consider the first two integrals on the left hand side. From Proposition \ref{fund}, we have
$$E_k(x,u,v)=H_k(x,u,v)A_{k,r},$$
where
$$A_{k,r}=-P_k^+R_k+\displaystyle\frac{2P_k^-T_k}{m+2k-4}.$$
For convenience, we let
$$E_k=E_k(x-y,u,v),\ F_k=F_k(x-y-u-v),\ H_k=H_k(x-y,u,v),$$
unless it is necessary to specify their dependence on the variables. Hence, the first two integrals can be rewritten as follows.
\be
&&\int\displaylimits_{\partial\Omega}\big(H_k(P_k^+-\frac{2P_k^-}{m+2k-4}),d\sigma_xR_kP_k^+f(x,u)\big)_u+\int\displaylimits_{\partial\Omega}\big(E_k,d\sigma_xP_k^+f(x,u)\big)_u\\
&=&\int\displaylimits_{\partial\Omega}\big(H_kP_k^+,d\sigma_xR_kP_k^+f(x,u)\big)_u-\frac{2}{m+2k-4}\int\displaylimits_{\partial\Omega}\big(H_kP_k^-,d\sigma_xR_kP_k^+f(x,u)\big)_u\\
&&-\int\displaylimits_{\partial\Omega}\big(H_kP_k^+R_k,d\sigma_xP_k^+f(x,u)\big)_u+\frac{2}{m+2k-4}\int\displaylimits_{\partial\Omega}\big(H_kP_k^-T_k,d\sigma_xP_k^+f(x,u)\big)_u.
\ee
Let $B_r=\{x:||x-y||<r\}\subset\Omega$, with some $r>0$, and $B_r^c=\Omega\backslash B_r$. Then we apply the Stokes' theorem for $T_k$ to the second integral and the Stokes' theorem for $R_k$ to the other three integrals. The previous equation becomes
\begin{eqnarray}
&=&\int\displaylimits_{\partial B_r}\big(H_kP_k^+,d\sigma_xR_kP_k^+f(x,u)\big)_u+\int\displaylimits_{B_r^c}\big(H_kP_k^+R_k,R_kP_k^+f(x,u)\big)_u\nonumber\\
&&+\int\displaylimits_{B_r^c}\big(H_kP_k^+,R_k^2P_k^+f(x,u)\big)_u-\frac{2}{m+2k-4}\bigg(\int\displaylimits_{\partial B_r}\big(H_kP_k^-,d\sigma_xR_kP_k^+f(x,u)\big)_u\nonumber\\
&&-\int\displaylimits_{B_r^c}\big(H_kP_k^-T_k,R_kP_k^+f(x,u)\big)_u-\int\displaylimits_{B_r^c}\big(H_kP_k^-,T_k^*R_kP_k^+f(x,u)\big)_u\bigg)\nonumber\\
&&-\int\displaylimits_{\partial B_r}\big(H_kP_k^+R_k,d\sigma_xP_k^+f(x,u)\big)_u-\int\displaylimits_{B_r^c}\big(H_kP_k^+R_k,R_kP_k^+f(x,u)\big)_u\nonumber\\
&&-\int\displaylimits_{B_r^c}\big(H_kP_k^+R_k^2,P_k^+f(x,u)\big)_u+\frac{2}{m+2k-4}\bigg(\int\displaylimits_{\partial B_r}\big(H_kP_k^-T_k,d\sigma_xP_k^+f(x,u)\big)_u\nonumber\\
&&+\int\displaylimits_{B_r^c}\big(H_kP_k^-T_kR_k,P_k^+f(x,u)\big)_u+\int\displaylimits_{B_r^c}\big(H_kP_k^-T_k,R_kP_k^+f(x,u)\big)_u\bigg).\nonumber
\end{eqnarray}
With four integrals cancelled above, it becomes
\begin{eqnarray}
&=&\int\displaylimits_{\partial B_r}\big(H_kP_k^+,d\sigma_xR_kP_k^+f(x,u)\big)_u-\frac{2}{m+2k-4}\int\displaylimits_{\partial B_r}\big(H_kP_k^-,d\sigma_xR_kP_k^+f(x,u)\big)_u\nonumber\\
&&-\int\displaylimits_{\partial B_r}\big(H_kP_k^+R_k,d\sigma_xP_k^+f(x,u)\big)_u+\frac{2}{m+2k-4}\int\displaylimits_{\partial B_r}\big(H_kP_k^-T_k,d\sigma_xP_k^+f(x,u)\big)_u\nonumber\\
&&+\int\displaylimits_{B_r^c}\big(H_kP_k^+,R_k^2P_k^+f(x,u)\big)_u-\frac{2}{m+2k-4}\int\displaylimits_{B_r^c}\big(H_kP_k^-,T_k^*R_kP_k^+f(x,u)\big)_u\nonumber\\
&&-\int\displaylimits_{B_r^c}\big(H_kP_k^+R_k^2,P_k^+f(x,u)\big)_u+\frac{2}{m+2k-4}\int\displaylimits_{B_r^c}\big(H_kP_k^-T_kR_k,P_k^+f(x,u)\big)_u.\label{equation2}
\end{eqnarray}
Recall that
\be
H_k(x-y,u,v)=\displaystyle\frac{(m+2k-4)\Gamma(\displaystyle\frac{m}{2}-1)}{4(4-m)\pi^{\frac{m}{2}}}||x-y||^{2-m}Z_k^2(\displaystyle\frac{(x-y)u(x-y)}{||x-y||^2},v)
\ee
and $P_k^+$ and $P_k^-$ are independent with respect to the variable $x$. Hence, from the homogeneity of $x-y$ in $H_k(x-y,u,v)$, we know that 
\be
\int\displaylimits_{\partial B_r}\big(H_kP_k^+,d\sigma_xR_kP_k^+f(x,u)\big)_u\longrightarrow 0,\ \int\displaylimits_{\partial B_r}\big(H_kP_k^-,d\sigma_xR_kP_k^+f(x,u)\big)_u\longrightarrow 0,
\ee
when $r$ approaches zero. Here, we give the details for the first integral approaching zero, the second can be derived from similar arguments. For convenience, we ignore the constant in $H_k(x-y,u,v)$. We also remind the reader that we will see similar statements later in this section.
\be
&&\int\displaylimits_{\partial B_r}\big(H_kP_k^+,d\sigma_xR_kP_k^+f(x,u)\big)_u\\
&&=\int\displaylimits_{\partial{B_r}}\int\displaylimits_{\Sm}||x-y||^{2-m}Z_k^2(\displaystyle\frac{(x-y)u(x-y)}{||x-y||^2},v)P_k^+n(x)R_kP_k^+f(x,u)dS(u)d\sigma(x).
\ee
Here, the outer normal vector $n(x)=\displaystyle\frac{y-x}{||x-y||}$. Let $x-y=r\zeta$, where $\zeta\in\Sm$. The equation above becomes
\be
&&\int\displaylimits_{\Sm}\int\displaylimits_{\Sm}r^{2-m}Z_k^2(\zeta u\zeta,v)P_k^+(-\zeta)R_kP_k^+f(y+r\zeta,u)r^{m-1}dS(u)dS(\zeta)\\
&=&-\int\displaylimits_{\Sm}\int\displaylimits_{\Sm}rZ_k^2(\zeta u\zeta,v)P_k^+\zeta R_kP_k^+f(y+r\zeta,u)dS(u)dS(\zeta),
\ee
where $r^{m-1}$ above comes from the Jacobian of the change of variable. Since $f(x,u)\in C^{\infty}(\Rm,\Hk)$ and $Z_k^2(u,v)$ is the reproducing kernel of $\Hk$, then 
$$Z_k^2(\zeta u\zeta,v)P_k^+\zeta R_kP_k^+f(y+r\zeta,u)$$ is bounded for $\zeta\in \Sm$ and $u\in\Sm$. Therefore, the previous integral goes to zero when $r$ goes to zero.\\
\par
On the other hand, from Lemma \ref{Ortho}, we observe that
$H_kP_k^-T_k\perp T_k^*R_kP_k^+f(x,u)$
with respect to $(\ ,\ )_u$. Therefore, equation (\ref{equation2}) becomes
\begin{eqnarray}
&=&\int\displaylimits_{\partial B_r}\big(H_k(-P_k^+R_k+\frac{2}{m+2k-4}P_k^-T_k),d\sigma_xP_k^+f(x,u)\big)_u\nonumber\\
&&+\int\displaylimits_{B_r^c}\big(H_kP_k^++H_kP_k^-,(R_k^2-\frac{2}{m+2k-4}T_k^*R_k)P_k^+f(x,u)\big)_u\nonumber\\
&&+\int\displaylimits_{B_r^c}\big(H_k(-P_k^+R_k^2+\frac{2}{m+2k-4}P_k^-T_kR_k),P_k^+f(x,u)\big)_u\nonumber\\
&=&\int\displaylimits_{\partial B_r}\big(H_kA_{k,r},d\sigma_xP_k^+f(x,u)\big)_u+\int\displaylimits_{B_r^c}\big(H_k,(R_k^2-\frac{2}{m+2k-4}T_k^*R_k)P_k^+f(x,u)\big)_u\nonumber\\
&&+\int\displaylimits_{B_r^c}\big(H_kA_{k,r}R_k,P_k^+f(x,u)\big)_u\nonumber\\
&=&\int\displaylimits_{\partial B_r}\big(E_k,d\sigma_xP_k^+f(x,u)\big)_u+\int\displaylimits_{B_r^c}\big(H_k,(R_k^2P_k^+-\frac{2}{m+2k-4}T_k^*R_kP_k^+)f(x,u)\big)_u.\label{firsttwoterms}
\end{eqnarray}
The last equation comes from
\be
H_k(x-y,u,v)A_{k,r}R_k=E_k(x-y,u,v)R_k=0,\ for\ x\in B_r^c.
\ee

Similar argument applies for the last two integrals on the left hand side in Theorem \ref{BorelPompeiu}. With
$$F_k(x-y,u,v)=H_k(x-y,u,v)B_{k,r},$$
where
$$B_{k,r}=-\displaystyle\frac{2P_k^+T_k^*}{m+2k-4}-\displaystyle\frac{(m+2k)P_k^-Q_k}{m+2k-4},$$
the last two integrals in Theorem \ref{BorelPompeiu} can be rewritten as follows.
\be
&&\int\displaylimits_{\partial\Omega}\big(H_k(\frac{2P_k^+}{m+2k-4}+\frac{m+2k}{m+2k-4}P_k^-),d\sigma_xQ_kP_k^-f(x,u)\big)_u+\int\displaylimits_{\partial\Omega}\big(F_k,d\sigma_xP_k^-f(x,u)\big)_u\\
&=&\frac{2}{m+2k-4}\bigg(\int\displaylimits_{\partial\Omega}\big(H_kP_k^+,d\sigma_xQ_kP_k^-f(x,u)\big)_u-\int\displaylimits_{\partial\Omega}\big(H_kP_k^+T_k^*,d\sigma_xP_k^-f(x,u)\big)_u\bigg)\\
&&+\frac{m+2k}{m+2k-4}\bigg(\int\displaylimits_{\partial\Omega}\big(H_kP_k^-,d\sigma_xQ_kP_k^-f(x,u)\big)_u-\int\displaylimits_{\partial\Omega}\big(H_kP_k^-Q_k,d\sigma_xP_k^-f(x,u)\big)_u\bigg).
\ee
Applying the Stokes' theorem for $Q_k$ and $T_k$ to the equation above, we have
\begin{eqnarray}
&=&\frac{2}{m+2k-4}\bigg[\int\displaylimits_{\partial B_r}\big(H_kP_k^+,d\sigma_xQ_kP_k^-f(x,u)\big)_u+\int\displaylimits_{B_r^c}\big(H_kP_k^+T_k^*,Q_kP_k^-f(x,u)\big)_u\nonumber\\
&&+\int\displaylimits_{B_r^c}\big(H_kP_k^+,T_kQ_kP_k^-f(x,u)\big)_u-\int\displaylimits_{\partial B_r}\big(H_kP_k^+T_k^*,d\sigma_xP_k^-f(x,u)\big)_u\nonumber\\
&&+\int\displaylimits_{B_r^c}\big(H_kP_k^+T_k^*Q_k,P_k^-f(x,u)\big)_u+\int\displaylimits_{B_r^c}\big(H_kP_k^+T_k^*,Q_kP_k^-f(x,u)\big)_u\bigg]\nonumber\\
&&+\frac{m+2k}{m+2k-4}\bigg[\int\displaylimits_{\partial B_r}\big(H_kP_k^-,d\sigma_xQ_kP_k^-f(x,u)\big)_u+\int\displaylimits_{B_r^c}\big(H_kP_k^-Q_k,Q_kP_k^-f(x,u)\big)_u\nonumber\\
&&+\int\displaylimits_{B_r^c}\big(H_kP_k^-,Q_k^2P_k^-f(x,u)\big)_u-\int\displaylimits_{\partial B_r}\big(H_kP_k^-Q_k,d\sigma_xP_k^-f(x,u)\big)_u\nonumber\\
&&+\int\displaylimits_{B_r^c}\big(H_kP_k^-Q_k^2,P_k^-f(x,u)\big)_u+\int\displaylimits_{B_r^c}\big(H_kP_k^-Q_k,Q_kP_k^-f(x,u)\big)_u\bigg].\nonumber
\end{eqnarray}
With four integrals cancelled above, the previous equation becomes
\begin{eqnarray}
&&\frac{2}{m+2k-4}\bigg[\int\displaylimits_{\partial B_r}\big(H_kP_k^+,d\sigma_xQ_kP_k^-f(x,u)\big)_u-\int\displaylimits_{\partial B_r}\big(H_kP_k^+T_k^*,d\sigma_xP_k^-f(x,u)\big)_u\nonumber\\
&&+\int\displaylimits_{B_r^c}\big(H_kP_k^+,T_kQ_kP_k^-f(x,u)\big)_u-\int\displaylimits_{B_r^c}\big(H_kP_k^+T_k^*Q_k,P_k^-f(x,u)\big)_u\bigg]\nonumber\\
&&+\frac{m+2k}{m+2k-4}\bigg[\int\displaylimits_{\partial B_r}\big(H_kP_k^-,d\sigma_xQ_kP_k^-f(x,u)\big)_u-\int\displaylimits_{\partial B_r}\big(H_kP_k^-Q_k,d\sigma_xP_k^-f(x,u)\big)_u\nonumber\\
&&+\int\displaylimits_{B_r^c}\big(H_kP_k^-,Q_k^2P_k^-f(x,u)\big)_u-\int\displaylimits_{B_r^c}\big(H_kP_k^-Q_k^2,P_k^-f(x,u)\big)_u\bigg].\label{equation3}
\end{eqnarray}
Similarly, from the homogeneity of $x-y$ in $H_k(x-y,u,v)$, we have
\be
\int\displaylimits_{\partial B_r}\big(H_kP_k^+,d\sigma_xQ_kP_k^-f(x,u)\big)_u\longrightarrow 0,\ \int\displaylimits_{\partial B_r}\big(H_kP_k^-,d\sigma_xQ_kP_k^-f(x,u)\big)_u\longrightarrow 0,
\ee
when $r$ approaches zero. From Lemma \ref{Ortho}, we have
\be
H_kP_k^+\perp Q_k^2P_k^-f(x,u),\ H_kP_k^-\perp T_kQ_kP_k^-f(x,u),
\ee
with respect to $(\ ,\ )_u$. Therefore, equation (\ref{equation3}) becomes
\begin{eqnarray}
&=&\int\displaylimits_{\partial B_r}\big(H_k(-\displaystyle\frac{2}{m+2k-4}P_k^+T_k^*-\displaystyle\frac{m+2k}{m+2k-4}P_k^-Q_k),d\sigma_xP_k^-f(x,u)\big)_u\nonumber\\
&&+\int\displaylimits_{B_r^c}\big(H_k(P_k^++P_k^-),(\frac{2T_kQ_k}{m+2k-4}+\frac{m+2k}{m+2k-4}Q_k^2)P_k^-f(x,u)\big)_u\nonumber\\
&&+\displaystyle\int\displaylimits_{B_r^c}\big(H_k(-\frac{2}{m+2k-4}P_k^+T_k^*Q_k-\frac{m+2k}{m+2k-4}P_k^-Q_k^2),P_k^-f(x,u)\big)_u\nonumber\\
&=&\int\displaylimits_{\partial B_r}\big(H_kB_{k,r},d\sigma_xP_k^-f(x,u)\big)_u+\int\displaylimits_{B_r^c}\big(H_k,(\frac{2T_kQ_k}{m+2k-4}+\frac{m+2k}{m+2k-4}Q_k^2)P_k^-f(x,u)\big)_u\nonumber\\
&&+\displaystyle\int\displaylimits_{B_r^c}\big(H_kB_{k,r}Q_k,P_k^-f(x,u)\big)_u\nonumber\\
&=&\int\displaylimits_{\partial B_r}\big(F_k,d\sigma_xP_k^-f(x,u)\big)_u+\int\displaylimits_{B_r^c}\big(H_k,(\frac{2T_kQ_k}{m+2k-4}+\frac{m+2k}{m+2k-4}Q_k^2)P_k^-f(x,u)\big)_u.\label{secondtwoterms}
\end{eqnarray}
The last equation comes from
\be
H_k(x-y,u,v)B_{k,r}Q_k=F_k(x-y,u,v)Q_k=0,\ for\ x\in B_r^c.
\ee
Combining (\ref{firsttwoterms}) and (\ref{secondtwoterms}), we have the left hand side in Theorem \ref{BorelPompeiu} is equal to
\be
&&\int\displaylimits_{\partial B_r}\big(E_k,d\sigma_xP_k^+f(x,u)\big)_u+\int\displaylimits_{B_r^c}\big(H_k,(R_k^2-\frac{2}{m+2k-4}T_k^*R_k)P_k^+f(x,u)\big)_u\\
&&+\int\displaylimits_{\partial B_r}\big(F_k,d\sigma_xP_k^-f(x,u)\big)_u+\int\displaylimits_{B_r^c}\big(H_k,(\frac{2T_kQ_k}{m+2k-4}+\frac{m+2k}{m+2k-4}Q_k^2)P_k^-f(x,u)\big)_u\\
&=&\int\displaylimits_{\partial B_r}\big(E_k,d\sigma_xP_k^+f(x,u)\big)_u+\int\displaylimits_{\partial B_r}\big(F_k,d\sigma_xP_k^-f(x,u)\big)_u-\int\displaylimits_{B_r^c}\big(H_k,\Dtwo f(x,u)\big)_u.
\ee
The last equation comes from (\ref{D21}). That is,
$$\Dtwo=-R_k^2P_k^++\frac{2T_k^*R_kP_k^+}{m+2k-4}-\frac{2T_kQ_kP_k^-}{m+2k-4}-\frac{(m+2k)Q_k^2P_k^-}{m+2k-4}.$$
 Now, we state that
\begin{eqnarray}
\int\displaylimits_{\partial B_r}\big(E_k,d\sigma_xP_k^+f(x,u)\big)_u=P_k^+f(y,v),\ 
\int\displaylimits_{\partial B_r}\big(F_k,d\sigma_xP_k^-f(x,u)\big)_u=P_k^-f(y,v),\label{one}
\end{eqnarray}
when $r$ approaches zero. If (\ref{one}) holds, then the previous equation becomes
\be
&=&P_k^+f(y,v)+P_k^-f(y,v)-\int\displaylimits_{B_r^c}\big(H_k,\Dtwo f(x,u)\big)_u\\
&=&f(y,v)-\int\displaylimits_{\Omega}\big(H_k,\Dtwo f(x,u)\big)_u,
\ee
which completes the proof of the theorem. The last equation comes from
\be
\int\displaylimits_{B_r}\big(H_k,\Dtwo f(x,u)\big)_u\longrightarrow 0,
\ee
when $r$ approaches zero because of the homogeneity of $x-y$ in $H_k(x-y,u,v)$. \\
\par
Now, we prove (\ref{one}). To accomplish this, we need the following lemma.
\begin{lemma}\cite{D}\label{lemmaharmonic}
Suppose $h_k:\ \Rm\longrightarrow \Clm$ is a harmonic polynomial homogeneous of degree $k$ and $m>2$. Suppose $u\in\Sm$ then
\be
\frac{1}{\omega_m}\int\displaylimits_{\Sm}h_k(xux)dS(x)=a_kh_k(u),
\ee
where $a_k=\displaystyle\frac{m-2}{m+2k-2}$.
\end{lemma}
We only provide the details for the first equation of (\ref{one}), similar argument also applies for the second equation. This argument can also be found in the proof of Theorem $7$ in \cite{D}. We rewrite
\be
&&\int\displaylimits_{\partial B_r}\big(E_k,d\sigma_xP_k^+f(x,u)\big)_u\\
&=&\int\displaylimits_{\partial B_r}\big(E_k,d\sigma_xP_k^+f(y,u)\big)_u+\int\displaylimits_{\partial B_r}\big(E_k,d\sigma_xP_k^+[f(x,u)-f(y,u)]\big)_u
\ee
Since the second integral on the right hand side tends to zero as $r$ goes to zero because of the continuity of $f(x,u)$ with respect to the variable $x$, we only need to deal with the first integral. We will need the property that 
$$E_k(x,u,v)=\displaystyle\frac{1}{\omega_ma_k}\displaystyle\frac{x}{||x||^m}Z_k^1(\displaystyle\frac{xux}{||x||^2},v)=\displaystyle\frac{1}{\omega_ma_k}Z_k^1(u,\displaystyle\frac{xvx}{||x||^2})\displaystyle\frac{x}{||x||^m},$$
more details can be found in \cite{D}. Hence, the first integral becomes
\newpage
\be
&&\int\displaylimits_{\partial B_r}\big(E_k,d\sigma_xP_k^+f(y,u)\big)_u\\
&=&\int\displaylimits_{\partial B_r}\int\displaylimits_{\Sm}\frac{1}{\omega_ma_k}Z_k^1(u,\frac{(x-y)v(x-y)}{||x-y||^2})\frac{x-y}{||x-y||^m}n(x)P_k^+f(y,u)dS(u)d\sigma(x)
\ee
where $n(x)$ is the unit outer normal vector and $d\sigma(x)$ is  the scalar measure on $\partial B_r$. Now $n(x)$ here is $\displaystyle\frac{y-x}{||x-y||}$. Hence the previous integral becomes
\be
&&\int\displaylimits_{\partial B_r}\int\displaylimits_{\Sm}\frac{1}{\omega_ma_k}Z_k^1(u,\frac{(x-y)v(x-y)}{||x-y||^2})\frac{x-y}{||x-y||^m}\frac{y-x}{||x-y||}P_k^+f(y,u)dS(u)d\sigma(x)\\
&=&\int\displaylimits_{\partial B_r}\frac{1}{r^{m-1}}\int\displaylimits_{\Sm}\frac{1}{\omega_ma_k}Z_k^1(u,\frac{(x-y)v(x-y)}{||x-y||^2})P_k^+f(y,u)dS(u)d\sigma(x)
\ee
By Lemma \ref{lemmaharmonic}, this integral is qual to
\be
\int\displaylimits_{\Sm}\frac{1}{\omega_ma_k}Z_k^1(u,v)P_k^+f(y,u)dS(u)=P_k^+f(y,v),
\ee
which completes the proof for (\ref{one}). Similar argument applies for the second equation of (\ref{one}) with the help of 
\be
F_k(x,u,v)=\frac{-1}{\omega_{m}a_k}u\frac{x}{||x||^m}Z_{k-1}^1(\frac{xux}{||x||^2},v)v=\frac{-1}{\omega_{m}a_k}uZ_{k-1}^1(u,\frac{xvx}{||x||^2})\frac{x}{||x||^m}v.
\ee
\end{proof}
As an application of the previous theorem, we have a Green type integral formula for the higher spin Laplace operator $\Dtwo$ immediately, when $\Dtwo f(x,u)=0$.
\begin{theorem}\textbf{(Green type integral formula)}\\
Suppose $\Omega$ and $\Omega'$ are defined as in the previous theorem. Let $f(x,u)\in C^{\infty}(\Rm,\Hk)$, $y\in\Omega$ and $\Dtwo f(x,u)=0$, then we have

\be
f(y,v)&=&\int\displaylimits_{\partial\Omega}\big(H_k(x-y,u,v)(P_k^+-\frac{2P_k^-}{m+2k-4}),d\sigma_xR_kP_k^+f(x,u)\big)_u\\
&&+\int\displaylimits_{\partial\Omega}\big(E_k(x-y,u,v),d\sigma_xP_k^+f(x,u)\big)_u\\
&&+\int\displaylimits_{\partial\Omega}\big(H_k(x-y,u,v)(\frac{2P_k^+}{m+2k-4}+\frac{m+2k}{m+2k-4}P_k^-),d\sigma_xQ_kP_k^-f(x,u)\big)_u\\
&&+\int\displaylimits_{\partial\Omega}\big(F_k(x-y,u,v),d\sigma_xP_k^-f(x,u)\big)_u.
\ee
\end{theorem}


\end{document}